\documentclass{article}

\usepackage[british]{babel}
\usepackage{amsthm}

\usepackage{amssymb, enumerate}
\usepackage{latexsym}
\usepackage{url}

\usepackage[all,cmtip]{xy}
\usepackage{amsmath,amssymb}

\newcounter{itemcounter}
\numberwithin{itemcounter}{section}

\newtheorem{thm}[itemcounter]{Theorem}
\newtheorem{lem}[itemcounter]{Lemma}
\newtheorem{defi}[itemcounter]{Definition}
\newtheorem{prop}[itemcounter]{Proposition}
\newtheorem{cor}[itemcounter]{Corollary}

\newtheorem{rem}[itemcounter]{Remark}

\newtheorem*{thm*}{Theorem}
\newtheorem*{con*}{Conjecture}
\newtheorem*{cor*}{Corollary}
\newtheorem*{ack*}{Acknowledgements}

\newcommand{\Res}{\mathop{\rm Res}\nolimits}
\newcommand{\Ind}{\mathop{\rm Ind}\nolimits}
\newcommand{\Inf}{\mathop{\rm Inf}\nolimits}

\newcommand{\Syl}{\mathop{\rm Syl}\nolimits}
\newcommand{\Irr}{\mathop{\rm Irr}\nolimits}
\newcommand{\IBr}{\mathop{\rm IBr}\nolimits}

\newcommand{\Hom}{\mathop{\rm Hom}\nolimits}

\newcommand{\Aut}{\mathop{\rm Aut}\nolimits}
\newcommand{\Ext}{\mathop{\rm Ext}\nolimits}
\newcommand{\Out}{\mathop{\rm Out}\nolimits}
\newcommand{\Inn}{\mathop{\rm Inn}\nolimits}

\newcommand{\Stab}{\mathop{\rm Stab}\nolimits}
\newcommand{\Br}{\mathop{\rm Br}\nolimits}

\newcommand{\Pic}{\mathop{\rm Pic}\nolimits}

\newcommand{\nth}{\mathop{\rm th}\nolimits}
\newcommand{\Id}{\mathop{\rm Id}\nolimits}

\newcommand{\rad}{\mathop{\rm rad}\nolimits}

\newcommand{\CF}{\mathop{\rm CF}\nolimits}
\newcommand{\prj}{\mathop{\rm prj}\nolimits}
\newcommand{\Perf}{\mathop{\rm Perf}\nolimits}
\newcommand{\Picent}{\mathop{\rm Picent}\nolimits}

\newcommand{\cT} {\mathcal{T}}
\newcommand{\cL} {\mathcal{L}}

\newcommand{\cO} {\mathcal{O}}
\newcommand{\cF} {\mathcal{F}}

\newcommand{\ZZ} {\mathbb{Z}}

\newcommand{\al}{\alpha}

\title{Some examples of Picard groups of blocks \footnote{This research was supported by the EPSRC (grant no. EP/M015548/1).} }
\author{Charles W. Eaton\footnote{School of Mathematics, University of Manchester, Manchester, M13 9PL, United Kingdom. Email: charles.eaton@manchester.ac.uk} and Michael Livesey\footnote{Friedrich-Schiller-Universit\"{a}t Jena, Fakult\"{a}t f\"{u}r Mathematik und Informatik, Institut f\"{u}r Mathematik, 07737 Jena, Germany. Email: michael.livesey@uni-jena.de}}
\date{22nd May 2019}

\begin{document}

\maketitle

\begin{abstract}
We calculate examples of Picard groups for $2$-blocks with abelian defect groups with respect to a complete discrete valuation ring. These include all blocks with abelian $2$-groups of $2$-rank at most three with the exception of the principal block of $J_1$. In particular this shows directly that all such Picard groups are finite and $\Picent$, the group of Morita auto-equivalences fixing the centre, is trivial. These are amongst the first calculations of this kind. Further we prove some general results concerning Picard groups of blocks with normal defect groups as well as some other cases.
\end{abstract}

\section{Introduction}

Let $\cO$ be a complete discrete valuation ring with $k:=\cO / J(\cO)$ algebraically closed of prime characteristic $p$. Let $K$ be the field of fractions of $\cO$, of characteristic zero. Let $G$ be a finite group and $B$ be a block of $\cO G$. We always assume that $K$ contains all $|G|^{\nth}$ roots of unity. The Picard group $\Pic(B)$ of $B$ consists of isomorphism classes of $B$-$B$-bimodules which induce $\cO$-linear Morita auto-equivalences of $B$. For $B$-$B$-bimodules $M$ and $N$, the group multiplication is given by $M \otimes_B N$. As yet relatively few examples have been calculated and there are many open questions regarding their structure, as raised in~\cite{bkl18}. One is that, whilst the Picard group of a $k$-block is usually infinite, it is not clear whether the Picard group of an $\cO$-block of a finite group must be finite. A related question is whether every element of $\Pic(B)$ can be taken to be a bimodule with endopermutation source. There are also no known examples of blocks $B$ where the subgroup $\Picent(B)$ of $\Pic(B)$ inducing the identity map on $Z(B)$ is nontrivial.

The purpose of this article is to find the Picard groups of some classes of examples, both to provide evidence for the main open questions, but also as tools for the classification of Morita equivalence classes. Previous examples where the Picard groups have been calculated are the blocks with cyclic or Klein four defect groups, and blocks of groups $P \rtimes E$ where $P$ is an abelian $p$-group with $E$ abelian and $[P,E]=P$, all in~\cite{bkl18}. Further it follows from~\cite{hk02} that $\Pic(B)$ is finite when $B$ is the unique block of a finite group $G$ with a self-centralizing normal $p$-subgroup.

In~\cite{el2018} and~\cite{wzz17} the Morita equivalence classes of $2$-blocks with abelian defect group of $2$-rank at most three were classified. Our examples of Picard groups include those for representatives of each Morita equivalence class of such blocks (except the principal block of $J_1$). Further we give general results in the case that the defect group of a block is abelian and normal.

Following~\cite{bkl18}, the groups $\cL(B)$ and $\cT(B)$ are the subgroups of $\Pic(B)$ of bimodules with linear and trivial sources respectively. For $n\geq1$, we denote by $G_n$ the group $(C_{2^n} \times C_{2^n})\rtimes C_3$, where the action of $C_3$ is given by permuting $a$, $b$ and $(ab)^{-1}$, where $a$ and $b$ are generators for the two cyclic factors. In particular, $G_1 \cong A_4$. Our main results are Theorem~\ref{thm:PicNorm} and Propositions~\ref{prop:sl2} and~\ref{Aut(SL_2(8)):prop}, which we summarise here.

\begin{thm}
\label{maintheorem}
Let $P$ be a finite abelian $2$-group.
\begin{enumerate}[(i)]
\item If $B=\cO (P\times G_n)$, where $n \geq 1$, then $\Pic(B) = \cL(B) \cong (P \rtimes \Aut(P)) \times(C_3\rtimes \Out(G_n))$;
\item If $B=\cO (G_{n_1} \times G_{n_2})$, where $n_1,n_2 \geq 1$, then $\Pic(B)=\cT(\cO (B)\cong$
\begin{enumerate}[(i)]
\item $(C_3\rtimes\Out(G_{n_1}))\wr S_2$ if $n_1=n_2$,
\item $(C_3\rtimes\Out(G_{n_1}))\times(C_3\rtimes \Out(G_{n_2}))$, if $n_1\neq n_1$;
\end{enumerate}
\item If $B=B_0(\cO (P \times A_5))$, then $\Pic(B) = \cL(B) \cong (P \rtimes \Aut(P))\times C_2$;
\item If $B=B_0(\cO (G_n \times A_5))$, where $n \geq 1$, then $\Pic(B) = \cT(B) \cong (C_3\rtimes\Out(G_n)) \times C_2$;
\item If $B=B_0(\cO ((C_{2^n})^3 \rtimes C_7))$, where $n \geq 1$, then $\Pic(B)=\cT(B) \cong C_7 \rtimes \Out((C_{2^n})^3 \rtimes C_7)$;
\item If $B=\cO ((C_{2^n})^3 \rtimes (C_7 \rtimes C_3)))$, where $n \geq 1$, then $\Pic(B) = \cT (B) \cong C_3 \rtimes \Out((C_{2^n})^3 \rtimes (C_7 \rtimes C_3))$;
\item If $B = B_0(\cO SL_2(2^n))$, where $n \geq 2$, then $\Pic(B) = \cT (B) \cong C_n$;
\item If $B = B_0(\cO \Aut(SL_2(8)))$, then $\Pic(B) = \cT(B) \cong C_3$.
\end{enumerate}

In each case $\Picent(B)=1$.
\end{thm}

\begin{rem}
Note that $\Aut(C_{2^n} \times C_{2^n}) \cong GL_2(\ZZ/2^n\ZZ) \cong O_2(GL_2(\ZZ/2^n\ZZ)) \rtimes S_3$. By Lemma \ref{lem:OutOut} we have $\Out(G_n) \cong N_{\Aut(O_2(G_n))}(C_3)/C_3$, so $\Out(G_n)$ is a $2$-group. 

Also $\Aut((C_{2^n})^3) \cong GL_3(\ZZ/2^n\ZZ) \cong O_2(GL_3(\ZZ/2^n\ZZ)) \rtimes GL_3(2)$. Now $GL_3(2)$ has a unique conjugacy class of subgroups $C_7 \rtimes C_3$, so it follows from the Schur-Zassenhaus theorem that $\Aut((C_{2^n})^3)$ also does. This uniquely defines $(C_{2^n})^3 \rtimes (C_7 \rtimes C_3)$ up to isomorphism. Similarly $(C_{2^n})^3 \rtimes C_7$ is uniquely defined.

By Lemma \ref{lem:OutOut} $\Out((C_{2^n})^3 \rtimes (C_7 \rtimes C_3)) \cong N_{GL_3(\ZZ/2^n\ZZ)}(C_7 \rtimes C_3)/(C_7 \rtimes C_3)$, which is a $2$-group, and $(C_{2^n})^3 \rtimes C_7 \cong N_{GL_3(\ZZ/2^n\ZZ)}(C_7)/(C_7)$, which is a $\{2,3\}$-group with Sylow $3$-subgroup of order three.
\end{rem}

\begin{cor}
\label{rank3corollary}
Let $G$ be a finite group and $B$ a $2$-block of $\cO G$ with abelian defect group $D$ of $2$-rank at most three, with the exception of $G=J_1$. Then the isomorphism type of $\Pic(B)$ is determined and $\Pic(B) \cong \Pic(C)=\cL(C)$ for some block $C$ (of $\cO H$ for some finite group $H$) Morita equivalent to $B$ with isomorphic defect group. In particular, $|\Pic(B)|$ is bounded in terms of $D$. Further $\Picent(B)=1$.
\end{cor}

\begin{rem}
Florian Eisele, in an as yet unpublished preprint, has shown that $\Pic(B_0(\cO J_1))=1$.
\end{rem}

It is observed in~\cite{eis18} that derived equivalence preserves finiteness of Picard groups. In particular, as by~\cite{el2018} Brou\'e's conjecture holds for $2$-blocks with abelian defect groups of $2$-rank at most three, finiteness can also be seen for such blocks by using Proposition \ref{finiteness:prop}.

The structure of the paper is as follows. In Section \ref{Picard_background} we define Picard groups and certain distinguished subgroups, and give some background results from~\cite{bkl18}. The groups of perfect self-isometries of blocks play a major role in much of this paper, and Section \ref{perfect_isometries} contains the relevant definitions and calculations. In Section \ref{normal_defect} we apply the results of the previous section together with Weiss' criterion to prove most of Theorem \ref{maintheorem}. In Section \ref{rank3} we calculate the Picard groups in the final cases of Theorem \ref{maintheorem}.

\section{Picard groups of blocks}
\label{Picard_background}

The following is based on~\cite{bkl18}. For further detail we also recommend~\cite{li18, li18b}. Let $G$ be a finite group and $B$ be a block of $\cO G$ with defect group $D$. Let $\cF$ be the fusion system for $B$ on $D$, defined using a maximal $B$-subpair $(D,b_D)$. Write $E=N_G(D,b_D)/DC_G(D)$, the inertial quotient. Write $\Aut(D,\cF)$ for the subgroup of $\Aut(D)$ of automorphisms stabilizing $\cF$. Write $\Out(D,\cF)=\Aut(D,\cF)/\Aut_\cF(D)$. Note that if $B$ is the principal block and $D$ is abelian, then $\Out(D,\cF) \cong N_{\Aut(D)}(\Aut_G(D))/\Aut_G(D)$, where $\Aut_G(D)$ is the subgroup of $\Aut(D)$ induced by conjugation in $N_G(D)$. 

A special case that will be useful later is the following:

\begin{lem}\label{lem:OutOut}
Let $D$ be an abelian $p$-group and $E$ a $p'$-subgroup of $\Aut(D)$. Write $G=D \rtimes E$ and let $\cF$ the fusion system associated to $\cO G$. Then $\Out(D,\cF)\cong N_{\Aut(D)}(E)/E\cong \Out(G)$.
\end{lem}

\begin{proof}
That $\Out(D,\cF)\cong N_{\Aut(D)}(E)/E$ follows from the definition and discussion in Section \ref{Picard_background}. Write $H=D \rtimes N_{\Aut(D)}(E)=N_{D \rtimes Aut(D)}(G)$. The natural homomorphism $\varphi:H \rightarrow \Aut(G)$ gives rise to a homomorphism $H/G \rightarrow \Out(G)$. Since $C_H(D)=D$ this must be injective. Now let $\alpha \in \Aut(G)$. Then $\alpha$ induces an automorphism, say $\beta$, of $D$ by restriction. Now $\varphi(\beta)\alpha^{-1} \in \Aut(G)$ fixes $D$ pointwise. Surjectivity will follow once we have shown that any $\psi \in \Aut(G)$ that fixes $D$ pointwise must be in $\Inn(G)$. By the Schur-Zassenhaus theorem $E$ and $\psi(E)$ are conjugate in $G$, and further they must then be conjugate in $D$, say by $n \in D$. Write $\tau_n \in \Aut(G)$ for the automorphism induced by conjugation by $n$, so $\tau_{n}^{-1}\psi(E)=E$. Now $\tau_{n}^{-1}\psi$ fixes $D$ pointwise. Let $g \in E$. Then $\tau_n^{-1}\psi(g) \in Dg$, since for all $x \in D$ we have $x^g=\tau_n^{-1}\psi(x^g)=x^{\tau_n^{-1}\psi(g)}$, so $\tau_n^{-1}\psi(g)g^{-1} \in C_G(D)=D$. Hence $\tau_{n}^{-1}\psi(g) \in Dg \cap E=\{ g \}$. We have shown that $\psi=\tau_n$, as claimed.

\end{proof}

Now let $A$ be a source algebra for $B$, so $A$ is a $D$-algebra and we may consider the fixed points $A^D$ under the action of $D$. Write $\Aut_D(A)$ for the group of $\cO$-algebra automorphisms of $A$ fixing each element of the image of $D$ in $A$, and $\Out_D(A)$ for the quotient of $\Aut_D(A)$ by the subgroup of automorphisms given by conjugation by elements of $(A^D)^\times$. As noted in~\cite{bkl18}, by~\cite[14.9]{pu88} $\Out_D(A)$ is isomorphic to a subgroup of $\Hom(E,k^\times)$. We shall sometimes refer to $\Out_D(B)$, which is the quotient of the group of $\cO$-algebra automorphisms of $B$ fixing each element of the image of $D$ in $B$, by the subgroup of automorphisms given by conjugation by elements of $(B^D)^\times$.

The Picard group $\Pic(B)$ of $B$ consists of isomorphism classes of $B$-$B$-bimodules which induce $\cO$-linear Morita auto-equivalences of $B$. For $B$-$B$-bimodules $M$ and $N$, the group multiplication is given by $M \otimes_B N$. Write $\cT(B)$ for the subset of $\Pic(B)$ consisting of bimodules with trivial source and $\cL(B)$ for the subset consisting of linear source modules. It is shown in~\cite{bkl18} that $\cT(B)$ and $\cL(B)$ form subgroups of $\Pic(B)$ and are described by the exact sequences
\begin{equation}\label{arr:exact}\begin{array}{lllllll}
1 & \rightarrow & \Out_D(A) & \rightarrow & \cT(B) & \rightarrow & \Out(D,\cF), \\
1 & \rightarrow & \Out_D(A) & \rightarrow & \cL(B) & \rightarrow & \Hom(D/\mathfrak{foc}(D),\mathcal{O}^\times) \rtimes \Out(D,\cF), \\
\end{array}\end{equation} where $\mathfrak{foc}(D)$ is the focal subgroup of $D$ with respect to $\mathcal{F}$, generated by the elements $\varphi(x)x^{-1}$ for $x \in D$ and $\varphi \in \Hom_\mathcal{F}(\langle x \rangle,D)$.

$\cO$-algebra automorphisms $\alpha$ of $B$ give rise to elements of the Picard group as follows. Define the $B$-$B$-bimodule ${}_\alpha B$ by taking ${}_\alpha B=B$ as sets and defining $a_1 \cdot m \cdot a_2 = \alpha(a_1)ma_2$ for $a_1,a_2,m \in B$. Inner automorphisms give isomorphic bimodules and $\alpha \mapsto {}_\alpha B$ gives rise to an injection $\Out(B) \rightarrow \Pic(B)$.

Each element of $\Pic(B)$ induces an automorphism of $Z(B)$. The subgroup consisting of those which induce the identity morphism is denoted $\Picent(B)$. Note $\Picent(B)$ is precisely the subgroup of bimodules in $\Pic(B)$ that fix every irreducible character.

The following is clear from~\cite{bkl18} but we state it here for convenience as it will be used frequently.

\begin{lem}
\label{Out_D(A)subgroup:lem}
Let $B$ be a block of $\cO G$ for a finite group $G$ and let $(D,b_D)$ be a maximal $B$-subpair. Suppose $N \lhd G$ with $G/N$ abelian, $C_G(D) \leq N$ and $G=N_G(D,b_D)N$. Then $\Out_D(B)$ has a subgroup isomorphic to $G/N$. Consequently, if $A$ is a source algebra for $B$, then $\Out_D(A)$ has a subgroup isomorphic to $G/N$.
\end{lem}

\begin{proof}
Let $\varphi$ be the inflation of an irreducible character of $G/N$. An element of $\Out_D(B)$ is realised by taking the bimodule ${}_{\al} B$  inducing the Morita equivalence given by the automorphism $\al$ of $B$ given by $\al (x)=\varphi(x)x$. It is clear from the permutation of $\Irr(G)$ given by ${}_\al B$ that distinct characters of $G/N$ give rise to distinct elements of $\Pic(B)$.
\end{proof}

In order to find $\cT(B)$ when $G$ is a direct product of groups we need to show that $\Out_D(A)$ factorises according to the factorisation of $G$.

Write $i$ for the identity element of the source algebra $A$, so $A=i\cO Gi$. As described in~\cite[Remark 1.2]{bkl18} elements of $\Out_D(A)$ correspond to direct summands of $\cO Gi \otimes _{\cO D} i \cO G$ as $B$-$B$-bimodules inducing Morita equivalences. In the following we will have to pass temporarily to the source algebra defined with respect to $k$ in order to apply~\cite[Lemma 10.37]{cr2}. We use the notation $kB$ (or $KB$) to denote $B \otimes_\cO k$ (or $B \otimes_\cO K$), and use similar notation for related objects.

\begin{lem}
\label{factorise_source_algebra:lemma}
Let $G_1$ and $G_2$ be finite groups, $B_i$ a block of $\cO G_i$ and $B=B_1 \otimes_{\cO} B_2$ a block of $\cO (G_1 \times G_2)$. Let $D$ be a defect group of $B$ with $D=D_1 \times D_2$, where $D_i$ is a defect group of $B_i$. Let $e_i$ be a source idempotent of $B_i$. Then $e:=e_1\otimes e_2$ is a source idempotent of $B$ and $eBe \cong e_1B_1e_1 \otimes_{\cO} e_2B_2e_2$.

Further $\Out_D(A) \cong \Out_{D_1}(A_1) \times \Out_{D_2}(A_2)$, where $A=eBe$ and $A_i=e_iB_ie_i$.
\end{lem}

\begin{proof}
We have $B^D=B_1^{D_1} \otimes_{\cO} B_2^{D_2}$ and $C_G(D)=C_G(D_1) \times C_G(D_2)$. Now~\cite[Lemma 10.37]{cr2} implies that $k(B_1^{D_1}e_1)\otimes_kk(B_2^{D_2}e_2)$ is an indecomposable $k(B_1^{D_1})\otimes_kk(B_2^{D_2})$-module. So the image of $e$ in $k(B^D)$ is primitive hence so is $e$ itself. In addition
$$\Br_{D_1\times D_2}(e)=\Br_{D_1}(e_1)\otimes_{\cO}\Br_{D_2}(e_2)\neq0$$
and so we have shown that $e$ is a source idempotent of $B$. Now every indecomposable $B$-$B$-summand of
$$\cO Ge \otimes _{\cO D} e\cO G \cong (\cO G_1e_1 \otimes_{\cO D_1} e_1\cO G_1) \otimes_{\cO} (\cO G_2e_2 \otimes _{\cO D_2} e_2\cO G_2)$$
is isomorphic to $M_1 \otimes_{\cO} M_2$, for an indecomposable $B_1$-$B_1$-summand $M_1$ of $\cO G_1e_1 \otimes _{\cO D_1} e_1\cO G_1$ and an indecomposable $B_2$-$B_2$-summand $M_2$ of $\cO G_2e_2 \otimes _{\cO D_2} e_2\cO G_2$ (since another application of~\cite[Lemma 10.37]{cr2} gives the analogous statement over $k$ and trivial source modules can be lifted uniquely to $\cO$). Finally we note that $K(M_1 \otimes_\cO M_2)$ induces a bijection of simple $KB$-modules if and only if each $KM_i$ induces a bijection of simple $KB_i$-modules. The claim now follows from~\cite[Th\'eor\`eme 1.2]{br90}.
\end{proof}


\section{Perfect isometries}
\label{perfect_isometries}

Before we calculate some Picard groups of blocks it will be necessary to determine some perfect self-isometry groups. We first introduce some notation.

Let $G$ be a finite group and $B$ a block of $\cO G$. We write $\Irr(G)$ (respectively $\Irr(B)$) for the set of irreducible characters of $G$ (respectively $B$), with respect to $K$. Write $G_{p'}$ for the set of $p$-regular elements of $G$, $\IBr(B)$ for the set of irreducible Brauer characters of $B$ and $\prj(B)$ for the set characters of projective indecomposable $B$-modules. $B_0(\cO G)$ will denote the principal block of $\cO G$.

\begin{defi}[\cite{br1990}]
We denote by $\CF(G,B,K)$ the $K$-subspace of class functions on $G$ spanned by $\Irr(B)$, by $\CF(G,B,\mathcal{O})$ the $\mathcal{O}$-submodule
\begin{align*}
\{\phi\in \CF(G,B,K):\phi(g)\in\mathcal{O}\text{ for all }g\in G\}
\end{align*}
of $\CF(G,B,K)$ and by $\CF_{p'}(G,B,\mathcal{O})$ the $\mathcal{O}$-submodule
\begin{align*}
\{\phi\in \CF(G,B,\mathcal{O}):\phi(g)=0\text{ for all }g\in G\backslash G_{p'}\}
\end{align*}
of $\CF(G,B,\mathcal{O})$.

Now in addition let $H$ be a finite group and $C$ a block of $\mathcal{O}H$. A \textbf{perfect isometry} between $B$ and $C$ is an isometry
\begin{align*}
I:\mathbb{Z}\Irr(B)\to\mathbb{Z}\Irr(C),
\end{align*}
such that
\begin{align*}
I_K:=I\otimes_{\mathbb{Z}}K:K\Irr(B)\to K\Irr(C),
\end{align*}
induces an $\cO$-module isomorphism between $\CF(G,B,\mathcal{O})$ and $\CF(H,C,\mathcal{O})$ and also between $\CF_{p'}(G,B,\mathcal{O})$ and $\CF_{p'}(H,C,\mathcal{O})$. (Note that by an isometry we mean an isometry with respect to the usual inner products on $\mathbb{Z}\Irr(B)$ and $\mathbb{Z}\Irr(C)$, so for all $\chi\in\Irr(B)$, $I(\chi)=\pm\psi$ for some $\psi\in\Irr(C)$).

If $H=G$ and $C=B$ then we describe $I$ as a perfect self-isometry of $B$. We denote by $\Perf(B)$ the group of perfect self-isometries of $B$.
\end{defi}

\begin{rem}
An alternative way of phrasing the condition that $I_K$ induces an isomorphism between $\CF_{p'}(G,B,\mathcal{O})$ and $\CF_{p'}(H,C,\mathcal{O})$ is that $I$ induces an isomorphism $\mathbb{Z}\prj (B) \cong \mathbb{Z}\prj (C)$.
\end{rem}

The following two well-known lemmas are both proved in~\cite{br1990}.

\begin{lem}\label{lem:isomcent}
Let $G$ and $G'$ be finite groups, $B$ and $B'$ blocks of $\cO G$ and $\cO G'$ respectively and $I:\mathbb{Z}\Irr(B)\to\mathbb{Z}\Irr(B')$ a perfect isometry. The $K$-algebra isomorphism between $Z(KB)$ and $Z(KB')$ given by the bijection of character idempotents induced by $I$ induces an $\cO$-algebra isomorphism $\phi_I:Z(B)\to Z(B')$.
\end{lem}

\begin{lem}\label{lem:morperf}
Any Morita equivalence of blocks induces a perfect isometry.
\end{lem}

Before proceeding with some specific examples we need a lemma about Picard groups and perfect self-isometry groups of group algebras of $p$-groups.

\begin{lem}\label{lem:perfP}
Let $P$ be a finite $p$-group. Then we have the following isomorphisms of groups.
\begin{enumerate}[(a)]
\item ${\Pic}(\cO P)\cong{\Out}(\cO P)\cong{\Hom}(P,\cO^\times)\rtimes{\Out}(P)\cong\cL(\cO P)$.
\item If $P$ is abelian, then ${\Perf}(\cO P)\cong{\Aut}(\cO P)\times C_2$.
\end{enumerate}
\end{lem}

\begin{proof}$ $
\begin{enumerate}[(a)]
\item Let $\lambda\in{\Hom}(P,\cO^\times)$. Then $M_{\Delta P}\uparrow^{P\times P}$ is the bimodule inducing the Morita equivalence given by tensoring with $\lambda$, where $M_{\Delta P}$ is the $\cO(\Delta P)$-module given by $(g,g).m=\lambda(g)m$, for all $g\in P$ and $m\in M_{\Delta P}$. Hence $\Hom (P,\cO^\times)$ embeds into ${\Pic}(\cO P)$, and the result follows from~\cite[Theorem 1.1]{bkl18}.

\item Since there is only one indecomposable projective module for $\cO P$, every perfect self-isometry of $\cO P$ must have all positive or all negative signs. Now by Lemma~\ref{lem:isomcent} the induced permutation of $\Irr(\cO P)$ induces an automorphism of $\Aut(Z(\cO P))= \Aut(\cO P)$. Hence the result.
\end{enumerate}
\end{proof}



We recall the character table of $A_4$, where we also set up some labelling of characters. We denote by $\omega\in\cO$ a primitive $3^{\operatorname{rd}}$ root of unity.

\begin{align*}
\begin{tabular}{|c||c|c|c|c|} \hline
 & $()$ & $(12)(34)$ & $(123)$ & $(132)$ \\ \hline
$\chi_1$ & $1$ & $1$ & $1$ & $1$ \\
$\chi_2$ & $1$ & $1$ & $\omega$ & $\omega^2$ \\
$\chi_3$ & $1$ & $1$ & $\omega^2$ & $\omega$ \\
$\chi_4$ & $3$ & $-1$ & $0$ & $0$ \\ \hline
\end{tabular}
\end{align*}

For the rest of this section we assume $p=2$.

\begin{prop}\cite[Proposition 2.8]{el2018}\label{prop:self_A4}
The perfect self-isometries of $\mathcal{O}A_4$ are precisely the isometries of the form:
\begin{align*}
I_{\sigma,\epsilon}:\mathbb{Z}\Irr(A_4)&\to\mathbb{Z}\Irr(A_4)\\
\chi_j&\mapsto\epsilon\delta_j\delta_{\sigma(j)}\chi_{\sigma(j)},
\end{align*}
for $1\leq j\leq 4$, where $\sigma\in S_4$, $\epsilon\in\{\pm1\}$ and $\delta_1=\delta_2=\delta_3=-\delta_4=1$. Hence ${\Perf}(\cO A_4)\cong S_4\times C_2$.
\end{prop}

Before proceeding we need a lemma. Let $n\in\mathbb{N}$ and $\zeta\in\cO$ a primitive $(2^n)^{\operatorname{th}}$ root of unity.

\begin{lem}\cite[Lemma 2.9]{el2018}\label{lem:roots_in_o}
Let $m$ be a positive integer and suppose $\sum_{i=0}^{2^m-1}\zeta^{l_i}\in2^m\mathcal{O}$, where $l_i\in\mathbb{Z}$ for $0\leq i<2^m$. Then either $\zeta^{l_0}=\dots=\zeta^{l_{2^m-1}}$ or $\sum_{i=0}^{2^m-1}\zeta^{l_i}=0$.
\end{lem}

Let $P$ be a finite abelian $2$-group.

\begin{thm}\label{thm:PA4}
Every perfect self-isometry of $\mathcal{O}(P\times A_4)$ is of the form $(J,I_{\sigma,\epsilon})$, where $J$ is a perfect isometry of $\cO P$ induced by an $\cO$-algebra automorphism, $\sigma\in S_4$ and $\epsilon\in\{\pm1\}$.
\end{thm}

\begin{proof}
We proceed as in the proof of~\cite[Theorem 2.11]{el2018}. The projective indecomposable characters are
\begin{align*}
\prj (\mathcal{O}(P\times A_4))=\{\chi_{P_1},\chi_{P_2},\chi_{P_3}\},\text{ where }\chi_{P_i}=\left(\sum_{\theta\in\Irr(P)}\theta\right)\otimes\left(\chi_i+\chi_4\right).
\end{align*}
Let $I$ be a perfect self-isometry of $\mathcal{O}(P\times A_4)$. Note that for $1\leq i\leq 3$, $\langle\chi_{P_i},\chi_{P_i}\rangle=2|P|$, where $\langle,\rangle$ is the usual inner product on $\ZZ\Irr(P\times A_4)$. Therefore, by considering elements of $\ZZ\prj(\mathcal{O}(P\times A_4))$ that also satisfy the above condition,
\begin{align}\label{align:im}
I(\chi_{P_i})=\pm\chi_{P_1},\pm\chi_{P_2},\pm\chi_{P_3},\pm(\chi_{P_1}-\chi_{P_2}),\pm(\chi_{P_1}-\chi_{P_3})\text{ or }\pm(\chi_{P_2}-\chi_{P_3}),
\end{align}
for $1\leq i\leq 3$. Now the characters in (\ref{align:im}) are transitively permuted by isometries of the form $(\Id,I_{\sigma,\epsilon})$ so we may assume that $I(\chi_{P_1})=\chi_{P_1}$. Next note that $\langle\chi_{P_i},\chi_{P_j}\rangle=|P|$ for $1\leq i\neq j\leq 3$ and so
\begin{align*}
I(\chi_{P_2})=\chi_{P_2},\chi_{P_3},\chi_{P_1}-\chi_{P_2}\text{ or }\chi_{P_1}-\chi_{P_3}.
\end{align*}
Therefore, by post-composing $I$ with $I_{\sigma,\epsilon}$ for
\begin{align*}
(\sigma,\epsilon)=(\Id,1),((23),1),((14),-1)\text{ or }((14)(23),-1)
\end{align*}
respectively, we may assume $I(\chi_{P_i})=\chi_{P_i}$ for $i=1,2$. The only $\chi$ in (\ref{align:im}) that satisfies $\langle\chi_{P_i},\chi\rangle=|P|$ for $i=1,2$ is $\chi_{P_3}$. So we have proved that $I(\chi_{P_i})=\chi_{P_i}$ for $1\leq i\leq 3$. Therefore, by considering $\langle\chi_{P_i},\chi\rangle$ for all $1\leq i\leq 3$ and $\chi\in\Irr(P\times A_4)$, there exist $J_m$ for $1\leq m\leq 4$, each a permutation of $\Irr(P)$, satisfying
\begin{align}\label{align:perm}
I(\theta\otimes\chi_m)=J_m(\theta)\otimes\chi_m,
\end{align}
for all $\theta\in\Irr(P)$. Until further notice we fix some $\theta\in\Irr(P)$. Note that
\begin{align*}
\frac{1}{12}\theta\otimes(\chi_1+\chi_2+\chi_3+3\chi_4)\in\CF(P\times A_4,\mathcal{O}(P\times A_4),\mathcal{O}).
\end{align*}
As $3$ is invertible in $\mathcal{O}$, this implies
\begin{align*}
\theta\otimes\left(\sum_{m=1}^4\delta_m\chi_m\right)\in4\CF(P\times A_4,\mathcal{O}(P\times A_4),\mathcal{O})
\end{align*}
(see Proposition~\ref{prop:self_A4} for the definition of $\delta_m$), and so
\begin{align}\label{align:im2}
I\left(\theta\otimes\left(\sum_{m=1}^4\delta_m\chi_m\right)\right)\in4\CF(P\times A_4,\mathcal{O}(P\times A_4),\mathcal{O}).
\end{align}
Evaluating (\ref{align:im2}) at $(x,1)$, $(x,(123))$ and $(x,(132))$, for some $x\in P$, and using (\ref{align:perm}) gives
\begin{align}
J_1(\theta)(x)+J_2(\theta)(x)+J_3(\theta)(x)+J_4(\theta)(x)&\in4\mathcal{O},\label{zeta1}\\
J_1(\theta)(x)+\omega J_2(\theta)(x)+\omega^2 J_3(\theta)(x)&\in4\mathcal{O},\label{zeta2}\\
J_1(\theta)(x)+\omega^2 J_2(\theta)(x)+\omega J_3(\theta)(x)&\in4\mathcal{O}\label{zeta3}.
\end{align}
Set $\zeta_m:=J_m(\theta)(x)$, for $1\leq m\leq 4$. Adding (\ref{zeta1}), (\ref{zeta2}) and (\ref{zeta3}) gives $3\zeta_1+\zeta_4\in4\mathcal{O}$. Now Lemma~\ref{lem:roots_in_o} tells us that $\zeta_1=\zeta_4$ as certainly $3\zeta_1+\zeta_4\neq0$. Therefore, by (\ref{zeta1}), $\zeta_2+\zeta_3\in2\mathcal{O}$. So again by Lemma~\ref{lem:roots_in_o} $\zeta_2=\zeta_3$ as $\zeta_2=-\zeta_3$ is prohibited by (\ref{zeta1}). Substituting into (\ref{zeta2}) gives $\zeta_1-\zeta_2\in4\mathcal{O}$. A final use of Lemma~\ref{lem:roots_in_o} tells us that $\zeta_1=\pm \zeta_2$ but $2\zeta_1\notin4\mathcal{O}$ and so we must have $\zeta_1=\zeta_2=\zeta_3=\zeta_4$.
\newline
\newline
We have shown that we may assume $I$ is of the form
\begin{align*}
I(\theta\otimes\chi_m)=J(\theta)\otimes\chi_m,
\end{align*}
for all $\theta\in\Irr(P)$ and $1\leq m\leq 4$, where $J$ is a permutation of $\Irr(P)$. In particular the $\mathcal{O}$-algebra automorphism of $Z(\mathcal{O}(P\times A_4))$ induced by $I$ leaves $\mathcal{O} P$ invariant. Therefore the permutation $J$ of $\Irr(P)$ must induce an automorphism of $\mathcal{O} P$ and the theorem is proved.
\end{proof}

For the remainder of this section we fix some $n\geq 1$. Set $H_n$ to be $(C_{2^n}\times C_{2^n})\rtimes S_3$, where the action of $S_3$ is given by permuting $a,b$ and $(ab)^{-1}$, where $a$ and $b$ are generators for the two cyclic factors. In addition set $G_n\leq H_n$ to be $(C_{2^n}\times C_{2^n})\rtimes C_3$, where the action of $C_3$ is given by cyclically permuting $a,b$ and $(ab)^{-1}$. For $1\leq i\leq 4$ set $\chi_i\in\Irr(G_n)$ to be the inflation to $G_n$ of the character of $A_4$ with the same label, and $\IBr(G_n)=\{\phi_1,\phi_2,\phi_3\}$ such that $\chi_i$ lifts $\phi_i$ for $1\leq i\leq 3$.

\begin{prop}\label{prop:PG}
Let $P$ be a finite abelian $2$-group and $n\geq 1$. Suppose $I$ is a permutation of $\Irr(P\times G_n)$ induced by a Morita auto-equivalence of $\mathcal{O}(P\times G_n)$. Then there exists $\sigma\in S_3$ and $J\in\Perf(\cO P)$, with all signs positive, such that $I$ satisfies $I(\theta\otimes\chi_i)=J(\theta)\otimes\chi_{\sigma(i)}$ for all $\theta\in\Irr(P)$ and $1\leq i\leq 3$.
\end{prop}

\begin{proof}
For $1\leq i\leq 3$ set $X_i:=\{\theta\otimes\chi_i|\theta\in\Irr(P)\}$. $I$ must permute the $X_i$'s as each $X_i$ is exactly the subset of $\Irr(P\times G_n)$ of characters that reduce to the Brauer character $\phi_i$. By composing with a Morita equivalence induced by tensoring with a linear character of $G_n$ and/or conjugation by some element of $H_n$, we may assume that $I$ leaves each $X_i$ invariant. Now for $1\leq i\leq 3$, we define $J_i:\Irr(P)\to\Irr(P)$ by
\begin{align*}
I(\theta\otimes\chi_i)=J_i(\theta)\otimes\chi_i.
\end{align*}
By Lemma~\ref{lem:morperf}, $I$ must be a perfect isometry. As
\begin{align*}
\sum_{\theta\in\Irr(P)}\alpha_\theta(\theta\otimes\chi_i)\in\CF(P\times A_4,\mathcal{O}(P\times A_4),\mathcal{O})
\end{align*}
if and only if
\begin{align*}
\sum_{\theta\in\Irr(P)}\alpha_\theta\theta\in\CF(P,\mathcal{O}P,\mathcal{O}),
\end{align*}
each $J_i$ must be a perfect isometry. Given that the same is true for $I$, $J_i$ must also have all signs positive. Now for each $\theta\in\Irr(P)$
\begin{align*}
\frac{1}{12}\theta\otimes(\chi_1+\chi_2+\chi_3+3\chi_4)\in\CF(P\times A_4,\mathcal{O}(P\times A_4),\mathcal{O}).
\end{align*}
Proceeding exactly as in the proof of Theorem~\ref{thm:PA4} proves that $J_1=J_2=J_3$ and hence the result is proved.
\end{proof}

Note that no non-trivial irreducible character of $C_{2^n}\times C_{2^n}$ is $G_n$-stable. Therefore for each $\chi\in\Irr(G_n)\backslash\{\chi_1,\chi_2,\chi_3\}$, $\chi$ reduces to the Brauer character $\phi_1+\phi_2+\phi_3$. Moreover, $|\Irr(G_n)\backslash\{\chi_1,\chi_2,\chi_3\}|=(2^{2n}-1)/3$.
\newline
\newline
Additionally we recall the decomposition matrix of $B_0(\cO A_5)$, where we also set up some labelling of characters and Brauer characters.

\begin{align*}
\begin{tabular}{|c||c|c|c|c|} \hline
& $\mu_1$ & $\mu_2$ & $\mu_3$ \\ \hline
$\psi_1$ & $1$ & $0$ & $0$ \\
$\psi_2$ & $1$ & $1$ & $0$ \\
$\psi_3$ & $1$ & $0$ & $1$ \\
$\psi_4$ & $1$ & $1$ & $1$ \\ \hline
\end{tabular}
\end{align*}

\begin{prop}\label{prop:GA5}
Let $n\geq 1$. Suppose $I$ is a permutation of $\Irr(B_0(\cO(G_n\times A_5)))$ induced by a Morita auto-equivalence of $B_0(\cO(G_n\times A_5))$. Then the set $\{\chi_i\otimes\psi_j|1\leq i\leq 3, 1\leq j\leq 4\}$ is left invariant by $I$.
\end{prop}

\begin{proof}
We first note that $I$ leaves $\{\chi_i\otimes\psi_j|1\leq i\leq 3, 1\leq j\leq 3\}$ invariant as this is exactly the subset of $\Irr(B_0(\cO(G_n\times A_5)))$ consisting of characters that reduce to a sum of $1$ or $2$ irreducible Brauer characters. Similarly
\begin{align*}
\{\chi_i\otimes\psi_4|1\leq i\leq 3\}\cup\{\alpha\otimes\psi_1|\alpha\in\Irr(G_n)\backslash\{\chi_1,\chi_2,\chi_3\}\}
\end{align*}
is exactly the subset of $\Irr(B_0(\cO(G_n\times A_5)))$ consisting of characters that reduce to a sum of $3$ distinct irreducible Brauer characters and so is also left invariant by $I$. Due to the comments proceeding the proposition, for $1\leq i\leq 3$
\begin{align*}
&\sum_{\substack{\phi\in\IBr(B_0(\cO(G_n\times A_5)))\\ \xi\in\Irr(B_0(\cO(G_n\times A_5)))}}D_{\chi_i\otimes\psi_4,\phi}D_{\xi,\phi}=\sum_{\substack{j=1\\ \xi\in\Irr(B_0(\cO(G_n\times A_5)))}}^3D_{\xi,\phi_i\otimes\mu_j}\\
=&\left(\sum_{\chi\in\Irr(\cO G_n)}d^G_{\chi,\phi_i}\right)\left(\sum_{\substack{j=1\\ \psi\in\Irr(B_0(\cO A_5))}}^3d^A_{\psi,\mu_j}\right)=\left(\frac{2^{2n}-1}{3}+1\right)8=8\frac{2^{2n}+2}{3}
\end{align*}
and for $\alpha\in\Irr(G_n)\backslash\{\chi_1,\chi_2,\chi_3\}$
\begin{align*}
&\sum_{\substack{\phi\in\IBr(B_0(\cO(G_n\times A_5)))\\ \xi\in\Irr(B_0(\cO(G_n\times A_5)))}}D_{\alpha\otimes\psi_1,\phi}D_{\xi,\phi}=\sum_{\substack{j=1\\ \xi\in\Irr(B_0(\cO(G_n\times A_5)))}}^3D_{\xi,\phi_j\otimes\mu_1}\\
=&\left(\sum_{\substack{j=1\\ \chi\in\Irr(\cO G_n)}}^3d^G_{\chi,\phi_j}\right)\left(\sum_{\psi\in\Irr(B_0(\cO A_5))}d^A_{\psi,\mu_1}\right)=3\left(\frac{2^{2n}-1}{3}+1\right)4=4(2^{2n}+2),
\end{align*}
where $D$, $d^G$ and $d^A$ are the decomposition matrices of $B_0(\cO(G_n\times A_5))$, $\cO G_n$ and $B_0(\cO A_5)$ respectively. However, since $I$ is induced by a Morita auto-equivalence,
\begin{align*}
\sum_{\substack{\phi\in\IBr(B_0(\cO(G_n\times A_5)))\\ \xi\in\Irr(B_0(\cO(G_n\times A_5)))}}D_{\chi\otimes\psi,\phi}D_{\xi,\phi}=\sum_{\substack{\phi\in\IBr(B_0(\cO(G_n\times A_5)))\\ \xi\in\Irr(B_0(\cO(G_n\times A_5)))}}D_{I(\chi\otimes\psi),\phi}D_{\xi,\phi},
\end{align*}
for all $\chi\in\Irr(\cO G)$ and $\psi\in\Irr(B_0(A_5))$. Therefore $\{\chi_i\otimes\psi_4|1\leq i\leq 3\}$ is left invariant by $I$ and the claim is proved.
\end{proof}

\section{Picard groups of blocks with normal defect group and related blocks}
\label{normal_defect}

We begin this section with some more notation. Let $G$ be a finite group and $B$ a block of $\cO G$. We denote by $e_B\in \cO G$ the block idempotent corresponding to $B$ and for each $\chi\in\Irr(G)$ we denote by $e_\chi\in KG$ the character idempotent corresponding to $\chi$. For this section we extend the meaning of $\Pic(B)$ and $\cT(B)$ to include the possibility that $B$ is a direct sum of blocks. In this case $\cT(B)$ denotes the subset of $\Pic(B)$ consisting of modules that are direct sums of trivial source modules.

An important tool in determining many of the Picard groups in this section will be Weiss' criterion, originally stated over the ring of $p$-adic integers. See~\cite[Remark 1.8]{bkl18} for a discussion of the generalisation of the ground ring. For all of this section up until Theorem \ref{thm:PicNorm} the prime $p$ may be taken to be arbitrary.

For $G$ a finite group, $H$ a subgroup and $M$ an $\cO G$-$\cO G$-bimodule, we define
\begin{align*}
{}^HM&=\{m\in M|g.m=m\text{ for all }g\in H\},\\
M^H&=\{m\in M|m.g=m\text{ for all }g\in H\}.
\end{align*}
In particular, if $M$ is a left $\cO G$-module we adopt the above notation by viewing $M$ as an $\cO G$-$\cO\{1\}$-bimodule.

\begin{prop}[Weiss~\cite{we88}]\label{prop:Weiss}
Let $P$ be a finite $p$-group, $M$ a finitely generated $\cO P$-module and $Q\lhd P$ such that $\Res_Q^P(M)$ is free and ${}^QM$ is a permutation $\cO(P/Q)$-module. Then $M$ is a permutation $\cO P$-module.
\end{prop}

Before we apply Weiss' criterion we need a proposition. Let $G$ be a finite group, $P$ a normal $p$-subgroup and $B$ a block of $\cO G$. We denote by $B^P$ the sum of blocks of $\cO(G/P)$ dominated by $B$, that is those blocks not annihilated by the image of $e_B$ under the natural $\cO$-algebra homomorphism $p_P:\cO G\to \cO(G/P)$, where $e_B\in\cO G$ is the block idempotent corresponding to $B$. Finally, we set
\begin{align*}
\Irr(B)^P:=\{\chi\in\Irr(B)|P\text{ is contained in the kernel of }\chi\}.
\end{align*}
Before we state our proposition we recall an important theorem of Brou\'e (see~\cite[Th\'eor\`eme 1.2]{br90}).

\begin{thm}\label{thm:brouechars}
Let $B$, $C$ be blocks of finite groups and $M$ a $B$-$C$-bimodule that's projective as both a left $B$-module and a right $C$-module. Then $M$ induces a Morita equivalence between $B$ and $C$ if and only if $KM$ induces a Morita equivalence between $KB$ and $KC$.
\end{thm}

\begin{prop}\label{prop:indmor}$ $
\begin{enumerate}[(a)]
\item The inflation map $\Inf:\Irr(G/P)\to\Irr(G)$ induces a bijection between $\Irr(B^P)$ and $\Irr(B)^P$.
\item Suppose $M$ is a $B$-$B$-bimodule inducing a Morita auto-equivalence of $B$ that permutes the elements of $\Irr(B)^P$. Then ${}^PM=M^P$ induces a Morita auto-equivalence of $B^P$. Furthermore, the permutation of $\Irr(B^P)$ induced by ${}^PM$ is equal to the permutation that $M$ induces on $\Irr(B)^P$, once these two sets have been identified using part (a).
\end{enumerate}
\end{prop}

\begin{proof}$ $
\begin{enumerate}[(a)]
\item This is~\cite[Lemma 8.6]{nt87}.
\item $M$ is projective as a left $\cO G$-module, so $M\downarrow_{S\times\{1\}}$ is a free $\cO S$-module, where $S\in\Syl_p(G)$. So ${}^PM$ is an $\cO$-summand (pure sublattice) of $M$ and ${}^PM$ is a free left $\cO(S/P)$-module. Hence ${}^PM$ is projective as a left $\cO(G/P)$-module. We have analogous statements for $M^P$ as a right $\cO(G/P)$-module. Now
\begin{align*}
KM\cong\bigoplus_{\chi\in\Irr(B)}V_\chi\otimes_KV_{f(\chi)}^*,
\end{align*}
where $V_\chi$ is the simple $KB$-module corresponding to $\chi$ and $f$ is the permutation of $\Irr(B)$ induced by $M$. So
\begin{align}\label{algn:charperm}
K({}^PM)={}^P(KM)\cong\bigoplus_{\chi\in\Irr(B)^P}V_\chi\otimes_KV_{f(\chi)}^*,
\end{align}
where the first equality follows from the fact that ${}^PM$ is an $\cO$-summand of $M$ and the isomorphism from the fact that for each $\chi\in\Irr(G)$, $\chi\downarrow_P$ is a sum of trivial or a sum of non-trivial irreducible characters of $P$. Also, since $f$ permutes $\Irr(B)^P$, we have ${}^P(KM)=(KM)^P$ and so, as ${}^PM$ and $M^P$ are $\cO$-summands of $M$, we have that ${}^PM=M^P$. We have now proved that $M^P$ is an $\cO(G/P)$-$\cO(G/P)$-bimodule, projective on the left and right, inducing a permutation of characters of $\Irr(B)^P$, which we can identify with $\Irr(B^P)$ by part (a). The claim now follows from Theorem~\ref{thm:brouechars}.
\newline
\newline
The claim about the induced permutation of $\Irr(B^P)$ follows from (\ref{algn:charperm}).
\end{enumerate}
\end{proof}

The following proposition is a consequence of Weiss' criterion and, together with Proposition~\ref{prop:indmor}, will be the main tool used in proving Theorem~\ref{thm:PicNorm}.

\begin{prop}
\label{Weiss_consequence:prop}
Suppose further to the hypotheses of Proposition~\ref{prop:indmor}(b) that ${}^PM\in\cT(B^P)$. Then $M\in\cT(B)$.
\end{prop}

\begin{proof}
We want to show that $M$ is an $\cO(S\times S)$-permutation module, where $S\in\Syl_p(G)$. Due to Proposition~\ref{prop:Weiss} it is enough to show that ${}^PM$ is an $\cO(S/P \times S)$-permutation module. However, this follows immediately from the fact that we already know ${}^PM$ is an $\cO(S/P\times S/P)$-permutation module.
\end{proof}

This allows us to generalize part of~\cite[Proposition 4.3]{bkl18}.

\begin{cor}
\label{decomp_number:cor}
Let $G$ be a finite group and $B$ a block of $\cO G$ with normal defect group $D\lhd G$. If
\begin{align*}
\Irr(B)^D=\{\chi\in\Irr(B)|\chi\text{ is a lift of some }\phi\in\IBr(B)\},
\end{align*}
then $\Pic(B)=\cT(B)$. In particular, if $D$ and $G/D$ are abelian and $Z(G)\cap D=\{1\}$, then $\Pic(B)=\cT(B)$.
\end{cor}

\begin{proof}
We first note that $|\Irr(B^D)|=|\IBr(B^D)|$ and so $B^D$ is a sum of defect zero blocks and so we trivially have that any $\Pic(B^D)=\cT(B^D)$. Now since
\begin{align*}
\{\chi\in\Irr(B)|\chi\text{ is a lift of some }\phi\in\IBr(B)\},
\end{align*}
is left invariant by any Morita auto-equivalence of $B$, we can apply Propositions~\ref{prop:indmor} and~\ref{Weiss_consequence:prop} and the first statement follows.
\newline
\newline
For the second statement let $P$ be an abelian $p$-group and $H$ a $p'$-subgroup of $\Aut(P)$. By~\cite[$\S5$ Theorem 2.3]{gor80} $C_P(H)$ has an $H$-invariant complement in $P$. Therefore if $C_P(H)$ is non-trivial then there exists a non-trivial, $H$-invariant, irreducible character of $P$. Note that $H$ naturally acts on $\Irr(P)$ and we can identify the action of $H$ on $\Irr(\Irr(P))$ with that of $H$ on $P$. Therefore we have the reverse implication as well, namely if there exists a non-trivial, $H$-invariant, irreducible character of $P$ then $C_P(H)$ is non-trivial.
\newline
\newline
Now we assume $D$ and $G/D$ are abelian and $Z(G)\cap D=\{1\}$. We first note that $p\nmid|G/D|$ as otherwise $D$ is properly contained in a normal $p$-subgroup of $G$, a contradiction. So, by the previous paragraph, no non-trivial character of $D$ is $G$-stable. Therefore, any $\chi\in\Irr(G)$ lying above any non-trivial $\lambda\in\Irr(D)$ has degree greater than $1$. So
\[\begin{array}{lll}
\Irr(B)^D&=&\{\chi\in\Irr(B)|\chi(1)=1\}\\
&=&\{\chi\in\Irr(B)|\chi\text{ is a lift of some }\phi\in\IBr(B)\}\\
\end{array}\]
and the claim follows from the first part of the corollary.
\end{proof}

Our first consequence is the boundedness of Picard groups for general blocks with normal defect groups.

\begin{prop}
\label{finiteness:prop}
Let $G$ be a finite group and $B$ a block of $\cO G$ with normal defect group $D$, then $\Picent(B)\leq\cT(B)$. In particular $|\Pic(B)|\leq|D|^2!|\cT(B)|$.
\end{prop}
\begin{proof}
As in Corollary~\ref{decomp_number:cor}, we have $|\Irr(B^D)|=|\IBr(B^D)|$ and so $\Pic(B^D)=\cT(B^D)$ is a direct sum of trivial source modules. Recall that $\Picent(B)$ is precisely the subgroup of bimodules in $\Pic(B)$ that fix every irreducible character. Therefore any $M\in\Picent(B)$ certainly satisfies the conditions of Proposition~\ref{prop:indmor}(b) with respect to $D$ and so the first part follows from Proposition~\ref{Weiss_consequence:prop}. By the Brauer-Feit theorem~\cite{bf59}, $\Picent(B)$ is a subgroup of $\Pic(B)$ of index at most $|D|^2!$ and the second part follows.
\end{proof}

Applying the methods of this section, we have for example:

\begin{thm}\label{thm:PicNorm}
Let $p=2$, $P$ be a finite abelian $2$-group and $n,n_1,n_2\in\mathbb{N}$.
\begin{enumerate}[(a)]
\item $\Pic(\cO (P \times G_n))=\cL(\cO (P \times G_n)) \cong (P \rtimes \Aut(P)) \times(C_3\rtimes \Out(G_n))$.
\item $\Pic(B_0(\cO (P \times A_5)))= \cL(B_0(\cO (P \times A_5))) \cong (P \rtimes \Aut(P)) \times C_2$.
\item $\Pic(\cO (G_{n_1}\times G_{n_2}))=\cT(\cO (G_{n_1}\times G_{n_2}))\cong$
\begin{enumerate}[(i)]
\item $(C_3\rtimes\Out(G_{n_1}))\wr S_2$ if $n_1=n_2$,
\item $(C_3\rtimes\Out(G_{n_1}))\times(C_3\rtimes\Out(G_{n_2}))$, if $n_1\neq n_1$;
\end{enumerate}
\item $\Pic(B_0(\cO (G_n\times A_5)))=\cT(B_0(\cO (G_n\times A_5))) \cong (C_3\rtimes\Out(G_n)) \times C_2$.
\item $\Pic(\cO((C_{2^n})^3\rtimes C_7))=\cT(\cO((C_{2^n})^3\rtimes C_7))\cong C_7\rtimes\Out((C_{2^n})^3\rtimes C_7)$.
\item $\Pic(\cO((C_{2^n})^3\rtimes(C_7\rtimes C_3)))=\cT(\cO((C_{2^n})^3\rtimes(C_7\rtimes C_3)))\cong C_3 \rtimes \Out((C_{2^n})^3\rtimes (C_7\rtimes C_3))$.
\end{enumerate}

For each block $B$ above $\Picent(B)=1$.
\end{thm}

\begin{proof}
Write $D$ for a defect group of the block $B$ of the group $G$ under consideration and $A$ for a source algebra. Let $E$ be the inertial quotient of the block and $\cF$ the fusion system associated to $B$. In the arguments that follow we will make free use of the results of~\cite{bkl18} as presented in Section \ref{Picard_background}, in particular the exact sequences in (\ref{arr:exact}). We also make repeated use of Lemma~\ref{lem:morperf} without further reference to it. When referring to $G_n$ it is assumed that the generators of $C_{2^n}\times C_{2^n}$ are $a$ and $b$ and that $C_3$ permutes $a$, $b$ and $(ab)^{-1}$ cyclically.

Note that the principal $2$-blocks of $G_n$, $A_5$ and of any $2$-group are equal to their source algebras, and so by Lemma \ref{factorise_source_algebra:lemma} the same is true for all blocks in (a)-(d).

\begin{enumerate}[(a)]
\item Let $M\in\Pic(B)$. Recall from $\S$\ref{perfect_isometries} that we are writing $\chi_1, \chi_2, \chi_3$ for the linear characters of $G_n$. By Proposition~\ref{prop:PG} and Lemma~\ref{lem:perfP}(b) we may compose $M$ with a Morita auto-equivalence induced by some element of $\Aut(\cO P)$ such that the induced permutation $I$ of $\Irr(P\times G_n)$ satisfies $I(\theta\otimes\chi_i)=\theta\otimes\chi_{\sigma(i)}$, where $\sigma\in S_3$, for all $\theta\in\Irr(P)$ and $1\leq i\leq 3$.
\newline
\newline
We now assume that $M$ does indeed induce the above permutation of characters, in particular $M$ satisfies the conditions of Proposition~\ref{prop:indmor}(b) with respect to $D$. Therefore, ${}^DM$ induces a Morita auto-equivalence of $\cO C_3$ and so certainly has trivial source. Proposition~\ref{Weiss_consequence:prop} now implies that $M$ must also have trivial source. We have shown that $\Pic(B)$ is generated by $\Aut(\cO P)$ and $\cT(B)$. In particular, Lemma~\ref{lem:perfP}(a) implies $\Pic(B)=\cL(B)$.
\newline
\newline
We now calculate $\cT(B)$. By the description proceeding Lemma~\ref{factorise_source_algebra:lemma} and considering the direct summands of $B\otimes_{\cO D}B$, we have that $\Out_D(A)\cong C_3$ is precisely the set of Morita auto-equivalences given by tensoring with the linear characters of $B$.
\newline
\newline
A direct calculation gives that $\Aut(D,\cF) \cong \Aut(P) \times N_{\Aut(O_2(G_n))}(C_3)$. Therefore, by Lemma~\ref{lem:OutOut}, $\Out(D,\cF) \cong \Aut(P)\times\Out(G_n)$ and so we can realise $\Out(D,\cF)$ as a quotient of $\cT(B)$. So $\cT(\cO (P \times G_n)) \cong \Aut(P)\times(C_3\rtimes\Out(G_n))$. We finally note that $\Aut(\cO P)\cap\cT(B)=\Aut(P)$ and the result follows.

\item Let $M\in\Pic(B)$. We first note that by~\cite[A1.3]{br1990}, $\cO A_4$ and $B_0(\cO A_5)$ are perfectly isometric. We can then apply Theorem~\ref{thm:PA4} and compose $M$ with an appropriate element of $\Aut(\cO P)$ to produce an element of $\Pic(B)$ that permutes $\Irr(B)^P$.
\newline
\newline
Suppose now that $M$ does permute $\Irr(B)^P$. By Proposition \ref{prop:indmor}(b) ${}^PM$ induces a Morita auto-equivalence of $B^P=B_0(\cO A_5)$. By~\cite[Theorem 1.5]{bkl18} $\Pic(B_0(\cO A_5))= \cT(B_0(\cO A_5)) \cong C_2$ and so by Proposition~\ref{Weiss_consequence:prop} $M$ must have trivial source. We have shown that $\Pic(B)$ is generated by $\Aut(\cO P)$ and $\cT(B)$. In particular, Lemma~\ref{lem:perfP}(a) implies $\Pic(B)=\cL(B)$.
\newline
\newline
We now calculate $\cT(B)$. By the proof of~\cite[Theorem 1.5]{bkl18}, or by observing that the nontrivial element of $\Pic(B_0(\cO A_5))$ comes from the conjugation action of $S_5$, we have $\Out_{C_2 \times C_2}(B_0(\cO A_5))=1$. Since $\Out_P(\cO P)$ is also trivial, by Lemma \ref{factorise_source_algebra:lemma} $\Out_D(A)=1$. It follows that $\cT(B) \leq \Out(D,\cF)\cong \Aut(P)\times C_2$. Note that we can realise all elements of $\Out(D,\cF)$ in $\cT(B)$. Indeed, $\Aut(P)$ is realised in the obvious way and the non-trivial element of $C_2$ is realised by conjugating by some $s\in S_5\backslash A_5$. We finally note that $\Aut(\cO P)\cap\cT(B)=\Aut(P)$ and the result follows.

\item That $\Pic(B)=\cT(B)$ follows immediately from the second part of Corollary~\ref{decomp_number:cor}. We now calculate $\cT(B)$. Note, by performing a similar calculation to that in part (a), $\Out_D(A)\cong C_3 \times C_3$ is precisely the set of Morita auto-equivalences given by tensoring with the linear characters of $B$. Now, by Lemma~\ref{lem:OutOut},
\begin{align*}
\Out(D,\cF)\cong\Out(G)\cong
\begin{cases}
\Out(G_{n_1})\wr S_2&\text{if }n_1=n_2,\\
\Out(G_{n_1})\times\Out(G_{n_2})&\text{if }n_1\neq n_2.
\end{cases}
\end{align*}
Note that $\Out(D,\cF)$ is realised as a quotient of $\cT(B)$ and the result follows.
\item Let $M \in \Pic(B)$ and let $S\in\Syl_p(G_n)$. By Proposition~\ref{prop:GA5} and Proposition \ref{prop:indmor}(b), ${}^SM$ induces a Morita auto-equivalence of $B^S \cong \cO C_3\otimes_\cO B_0(\cO A_5)$. We now claim that $\Pic(B^S)=\cT(B^S)$. We have
\begin{align*}
B^S=\bigoplus_{\lambda\in\Irr(C_3)}e_\lambda B^S
\end{align*}
and each $e_\lambda B^S\cong B_0(\cO A_5)$. Let $N\in\Pic(B^S)$. Certainly $\Pic(\cO C_3)=\cT(\cO C_3)$ acts as the full permutation group on $\Irr(C_3)$. So we may assume
\begin{align*}
N=\bigoplus_{\lambda\in\Irr(C_3)}N_\lambda\in\Pic(B^S),
\end{align*}
where each $N_\lambda\in\Pic(e_\lambda B^S)$. Now, as in (b), each $N_\lambda$ has trivial source. Therefore ${}^SM\in\Pic(B^S)=\cT(B^S)$ and hence, by Theorem~\ref{Weiss_consequence:prop}, $M\in\cT(B)$.
\newline
\newline
We now calculate $\cT(B)$. Once again, by Lemma \ref{factorise_source_algebra:lemma} we have that $\Out_D(A) \cong C_3$, which must then be the set of Morita auto-equivalences given by tensoring with the linear characters of $\cO G_n$. Now, a direct calculation when $n=1$ or by noting that $\Aut(D,\cF)$ must respect the decomposition $D=(D\cap G_n)\times (D\cap A_5)$ and applying Lemma~\ref{lem:OutOut} when $n>1$,
\begin{align}\label{algn:Out}
\Out(D,\cF)\cong
\begin{cases}
C_2 \wr C_2 \cong D_8 &\text{if }n=1,\\
\Out(G_n)\times C_2&\text{if }n>1.
\end{cases}
\end{align}
If $n>1$ then, once again, we can realise $\Out(D,\cF)$ as a quotient of $\cT(B)$. The result follows for $n>1$. Now suppose $n=1$. We have $S_3 \times C_2 \leq \cT(B)$, where $\Out_D(A) \cong C_3$ consists of elements acting as multiplication by a linear character of $A_4$ and the remaining elements come from conjugation by elements of $S_4 \times S_5$. Now either $S_3 \times C_2\cong\cT(B)$ or $S_3 \times C_2$ is isomorphic to an index $2$ subgroup $N$ of $\cT(B)$. In the latter case the isomorphism in (\ref{algn:Out}) gives that the element $\alpha$ of $\Pic(B)$ induced by conjugation by some $s\in S_5\backslash A_5$ is not central in $\Pic(B)$. However, $Z(N)=\langle \alpha\rangle$ and since $N$ has index $2$ in $\Pic(B)$, $\langle \alpha\rangle\leq Z(\Pic(B))$, a contradiction. Therefore we do indeed have that $\cT(B)\cong S_3\times C_2$.

\item That $\Pic(B)=\cT(B)$ follows immediately from Corollary~\ref{decomp_number:cor}. By the description preceding Lemma~\ref{factorise_source_algebra:lemma}, $\Out_D(A)\cong C_7$. Lemma~\ref{lem:OutOut} gives $\Out(D,\cF)\cong\Out(G)$ and the result follows from (\ref{arr:exact}), noting that as $\Out(D,\cF)$ is realised by a subgroup of $Pic(B)$ (given by group automorphisms) we do have a semidirect product. 

\item Again $\Pic(B)=\cT(B)$ follows immediately from Corollary~\ref{decomp_number:cor}. To calculate $\Out_D(A)$, consider
\begin{align*}
H:=C_{2^n}^3\rtimes C_7\leq C_{2^n}^3\rtimes(C_7\rtimes C_3).
\end{align*}
Then
\begin{align*}
\cO H\otimes_{\cO D}\cO H\cong\bigoplus_{\lambda\in\Irr(C_7)}M_\lambda,
\end{align*}
where $M_\lambda\in\Pic(\cO H)$ is given by tensoring with $\lambda$. Note that
\begin{align*}
\Stab_{G\times G}(M_\lambda)\cong
\begin{cases}
(\Delta G).(H\times H)&\text{if }\lambda=1,\\
(H\times H)&\text{if }\lambda\neq 1.
\end{cases}
\end{align*}
Therefore the only summands of $\cO G\otimes_{\cO D}\cO G$ that might possibly induce a Morita auto-equivalence on $\cO G$ are the direct summands of
\begin{align*}
M_1\uparrow^{G\times G}\cong \bigoplus_{\mu\in\Irr(C_3)}N_\mu,
\end{align*}
where $N_\mu\in\Pic(\cO G)$ is given by tensoring with $\mu$. In particular $\Out_D(A)\cong C_3$. Lemma~\ref{lem:OutOut} gives $\Out(D,\cF)\cong \Out(G)$ and so the result follows from (\ref{arr:exact}), noting as above that we do have a semidirect product since $\Out(D,\cF)$ is realised by a subgroup of $Pic(B)$. 
\end{enumerate}

The final remark concerning $\Picent(B)$ follows since in each case every element of $\Pic(B)$ induces a nontrivial permutation of $\Irr(B)$.
\end{proof}

\section{Blocks with abelian defect groups of $2$-rank at most three}
\label{rank3}

In~\cite{el2018}, which uses also results of~\cite{wzz17}, the Morita equivalence classes of $2$-blocks with abelian defect groups of $2$-rank at most three were classified. In this section we complete the computation of the Picard groups of a representative from each Morita equivalence class, so that the isomorphism type of the Picard group is known for every such block.

\begin{prop}[\cite{el2018},\cite{wzz17}]
Let $G$ be a finite group and $B$ a block of $\mathcal{O}G$ with defect group $D$ of $2$-rank at most $3$. Then $B$ is Morita equivalent to the principal block of one of:
(i) $\cO D$;
(ii) $\cO (D \rtimes C_3)$;
(iii) $\cO (C_{2^n}\times A_5)$ for $n \geq 0$;
(iv) $\cO (D \rtimes C_7)$;
(v) $\cO SL_2(8)$;
(vi) $\cO (D \rtimes (C_7 \rtimes C_3))$;
(vii) $\cO J_1$;
(viii) $\cO \Aut(SL_2(8))$.
\end{prop}

In order to prove Corollary \ref{rank3corollary}, following Lemma \ref{lem:perfP} and Theorem \ref{thm:PicNorm} it remains to calculate the Picard groups for $SL_2(8)$ and $\Aut(SL_2(8))$. As mentioned earlier, the Picard group of the principal block of $J_1$ has been calculated by Eisele (and is trivial).

\subsection{The principal $2$-block of $SL_2(2^n)$}

\begin{lem}
\label{SL2(2^n)_Borel:lem}
Let $n\in\mathbb{N}$ and $G=(C_2)^n \rtimes C_{2^n-1}$, where $C_{2^n-1}$ acts transitively on the non-trivial elements of $(C_2)^n$. Then $\Pic(\cO((C_2)^n \rtimes C_{2^n-1})) = \cT(\cO((C_2)^n \rtimes C_{2^n-1})) \cong C_{2^n-1} \rtimes C_n$, where the elements of $C_{2^n-1}$ permute the linear characters of $G$. Further $\Picent(\cO G)=1$.
\end{lem}

\begin{proof}
This follows from~\cite[4.3]{bkl18}. That $\Picent(\cO G)=1$ follows since every nontrivial element of $\Pic(\cO G)$ induces a nontrivial permutation of $\Irr(G)$.
\end{proof}

The calculation of the Picard group of the principal $2$-block of $SL_2(2^n)$ is a generalisation of the arguments in~\cite{bkl18} for $A_5 \cong SL_2(4)$.

\begin{prop}\label{prop:sl2}
Let $B=B_0(\mathcal{O}SL_2(2^n))$, where $n>1$. Then $\Pic(B) = \cT(B) \cong C_n$ and $\Picent(B)=1$.
\end{prop}

\begin{proof}
Let $G=SL_2(2^n)$, where $n > 1$, and let $B$ be the principal block of $\mathcal{O}G$. Write $q=2^n$. Let $D \in \Syl_2(G)$, so $D$ is elementary abelian of order $q$ and $H:=N_G(D)=D \rtimes E$ where $E \cong C_{q-1}$. Let $b=\cO H$. By Lemma \ref{SL2(2^n)_Borel:lem} $\Pic(b) = \cT(b)=C_{q-1} \rtimes C_n$, where the elements of $C_{q-1}$ permute the linear characters of $H$ and $C_n$ acts as group automorphisms obtained from $\Aut(G)$ (note that $N_{\Aut(G)}(D) \cong H \rtimes C_n$). Now $D$ is a trivial intersection subgroup of $G$, so induction and restriction gives a splendid stable equivalence of Morita type $F: \mod (b) \rightarrow \mod(B)$. Let $M \in \Pic(B)$. By~\cite{ko94} the trivial module for $B \otimes_\cO k$ is distinguished as the unique simple module $V$ with $\Ext_{kG}^1(V,W) \neq 0$ for three distinct simple $kG$-modules $W$. It follows that $M_H:=F^* \circ M \circ F:\mod(b) \rightarrow \mod(B)$ gives a stable auto-equivalence of Morita type preserving the trivial $b$-module. By~\cite[Corollary 3.3]{cr00} $M_H$ must then preserve every simple $b$-module, and so it is a Morita equivalence. The subgroup of $\Pic(b)$ consisting of bimodules fixing the trivial module is the subgroup of bimodules induced by group automorphisms of $G$ is described above, and is isomorphic to $C_n$. Since $M_H \in \cT(b)$ and we have a splendid stable equivalence between $b$ and $B$, it follows that $M \in \cT(B)$. We now apply the description of $\cT(B)$ in (\ref{arr:exact}). Let $A$ be a source algebra for $B$. Suppose that $M \in \Out_D(A)$ is nontrivial. Then, as $F$ is induced by a trivial source module, $M_H \in \Out_D(\cO H)$ and is also nontrivial. Hence $M_H$ permutes the simple modules of $H$, with no fixed points. In particular $M_H$ does not fix the trivial module, a contradiction. Hence $\Out_D(A)=1$. Now $\Out(D,\mathcal{F}) \cong C_n$. Since $\Out(G) \cong C_n$, the result follows, noting that every such automorphism acts nontrivially on $\Irr(B)$ so $\Picent(B)=1$.
\end{proof}

\subsection{The principal $2$-block of $\Aut(SL_2(8))$}

The calculation of $\Pic(B_0(\cO\Aut(SL_2(8))))$ is complicated by the fact that we may no longer use~\cite[Corollary 3.3]{cr00}, which requires the inertial quotient to be cyclic. Instead we must show directly that any Morita equivalence of $B_0(\cO\Aut(SL_2(8)))$ restricted to the principal block $b$ of the normalizer of a Sylow $2$-subgroup gives rise to a auto-equivalence of $b$ permuting the simple modules.

\begin{prop}
\label{Aut(SL_2(8)):prop}
Let $B=B_0(\cO \Aut(SL_2(8)))$ or $B_0(\cO {}^2G_2(3^{2m+1}))$ for $m \geq 1$. Then $\Pic(B) = \cT(B) \cong C_3$ and $\Picent(B)=1$.
\end{prop}

\begin{proof}
Let $G=\Aut(SL_2(8))$ and $N \leq G$ with $N \cong SL_2(8)$. Let $D \in \Syl_2(G)$ and write $H=N_G(D)$. We have $H \cong D \rtimes (C_7 \rtimes C_3)=(D \rtimes C_7) \rtimes C_3$. Write $B$ for the principal block of $\cO G$ and $b$ for the principal block of $\cO H$.

Let $M \in \Pic(B)$. Since $D$ is a trivial intersection subgroup of $G$, induction and restriction gives a splendid stable equivalence of Morita type $F:\mod(b) \rightarrow \mod(B)$. Hence $M_H:=F^* \circ M\circ F:\mod(b) \rightarrow \mod(b)$ induces a stable auto-equivalence of Morita type of $b$. We show that this stable equivalence sends simple modules to simple modules.

Using GAP, with output given in~\cite{GAP_decomp}, $B$ has irreducible characters of degrees $1,1,1,7,7,7,21,27$, irreducible Brauer characters of degrees $1,1,1,6,12$ (which we write $I$, $1$, $1^*$, $6$, $12$) and decomposition matrix
$$\left( \begin{array}{ccccc}
1 & 0 & 0 & 0 & 0 \\
0 & 1 & 0 & 0 & 0 \\
0 & 0 & 1 & 0 & 0 \\
1 & 0 & 0 & 1 & 0 \\
0 & 1 & 0 & 1 & 0 \\
0 & 0 & 1 & 1 & 0 \\
1 & 1 & 1 & 1 & 1 \\
1 & 1 & 1 & 2 & 1 \\
\end{array} \right).$$

From examination of the decomposition matrix, $M \otimes k$ permutes the modules of dimension one, fixing the other simple modules.

$b$ has simple modules of dimensions $1,1,1,3,3$, labelled $1_1,1_2,1_3,3_1,3_2$. We may choose our labelling so that $I$, $1$, $1^*$ has Green correspondent $1_1$, $1_2$, $1_3$ respectively. By~\cite[Theorem 5.3]{lm80} $\Res_H^G(6)= \begin{array}{c} 3_1 \\ 3_2 \\ \end{array}$ and
$$\Res_H^G(12)= \begin{array}{ccc}
 & 3_2 & \\
3_2 & & 3_1 \\
 & 3_1 & \\
 \end{array},$$ where the Loewy and socle series coincide. By Frobenius reciprocity $\Ind_H^G(3_1)$ has head $12$ and socle $6$, and $\Ind_H^G(3_2)$ has head $6$ and socle $12$.

Write $P_S$ for the projective cover of the simple module $S$. By~\cite[Theorem 4.1]{lm80} $P_{12}$ has Loewy and socle series
$$\begin{array}{ccc}
 & 12 & \\
 & 6 & \\
1 & I & 1^* \\
 & 6 & \\
1 & I & 1^* \\
 & 6 & \\
 & 12 & \\
 \end{array}.$$
Since $P_{12}$ is the projective cover of $\Ind_H^G(3_1)$ and injective hull of $\Ind_H^G(3_2)$, and since both induced modules have dimension $27$, we must have $\Ind_H^G(3_1) \cong P_{12} / \rad^4(P_{12})$ and $\Ind_H^G(3_2) \cong \rad^3(P_{12})$. We deduce that $M$ fixes the Green correspondents of $3_1$ and $3_2$, and permutes those of the linear modules, so $M_H$ permutes the simple modules. It follows by~\cite[Theorem 4.14.10]{li18} that $M_H$ induces a Morita auto-equivalence of $b$. By Theorem \ref{thm:PicNorm} $\Pic(b)=\cT(b)$, so $M_H$ has trivial source as a bimodule. Since $F$ is induced by a trivial source bimodule it follows that $M$ also has trivial source, i.e., $M \in \cT(B)$.

We now calculate $\cT(B)$. Let $A$ be a source algebra for $B$. Now $\Out_D(A)$ is isomorphic to a subgroup of $\Hom_k(C_7 \rtimes C_3,k^\times) \cong C_3$. It follows from Lemma \ref{Out_D(A)subgroup:lem} that $\Out_D(A) \cong C_3$. Since $C_7 \rtimes C_3$ is a non-normal maximal subgroup of $GL_3(2)$ it follows that $\Out(D,\mathcal{F})=N_{GL_3(2)}(C_7 \rtimes C_3)/(C_7 \rtimes C_3)=1$. Hence $\cT(B) \cong C_3$ as required. Note that each element of $\Pic(B)$ induces a nontrivial permutation of $\Irr(B)$, so $\Picent(B)=1$.

Now suppose that $B=B_0(\cO {}^2G_2(3^{2m+1}))$ for some $m \geq 1$. By~\cite[Example 3.3]{ok97} there is a splendid Morita equivalence between $B_0(\cO {}^2G_2(3^{2m+1}))$ and $B_0(\Aut(SL_2(8))$, i.e., one induced by a trivial source bimodule, so that $\cT(B_0(\cO {}^2G_2(3^{2m+1}))) \cong \cT(B_0(\cO \Aut(SL_2(8))))$ and $\Pic(B_0(\cO {}^2G_2(3^{2m+1}))) \cong \Pic(B_0(\cO \Aut(SL_2(8))))$
\end{proof}


\begin{ack*} We thank the referee for their careful reading of the manuscript and for their helpful comments and suggestions.
\end{ack*}

\end{document}